\documentclass[a4paper,12pt]{amsart}
\usepackage[utf8x]{inputenc}
\usepackage[T1]{fontenc}
\usepackage[a4paper,includeheadfoot,margin=2.54cm]{geometry}
\usepackage{stmaryrd}
\usepackage{amsmath}
\usepackage{amssymb}
\usepackage{amsthm}
\usepackage{mathrsfs}
\usepackage[all]{xy}
\usepackage{xcolor}
\usepackage{enumitem}
\setenumerate[1]{itemsep=0pt,partopsep=0pt,parsep=\parskip,topsep=5pt}
\setitemize[1]{itemsep=0pt,partopsep=0pt,parsep=\parskip,topsep=5pt}
\setdescription{itemsep=0pt,partopsep=0pt,parsep=\parskip,topsep=5pt}
\usepackage{times}
\usepackage{mathabx}
\usepackage{amsrefs}
\usepackage[pdftex]{graphicx}
\usepackage[pdftex,
colorlinks=true,
linkcolor=purple,
citecolor=blue,
hyperindex,
plainpages=false,
bookmarks=true,
bookmarksopen=true,
bookmarksnumbered]{hyperref}
\theoremstyle{plain}
\newtheorem{theorem}{Theorem}[section]
\newtheorem*{theorem*}{Theorem}

\newtheorem{lemma}[theorem]{Lemma}

\theoremstyle{definition}

\theoremstyle{example}

\theoremstyle{note}

\theoremstyle{remark}

\numberwithin{equation}{section}

\newcommand{\HM}{\mathcal{H}}
 
\newcommand{\mass}{\mathsf{M}} 
\newcommand{\fn}{\mathsf{W}} 
\newcommand{\RC}{\mathscr{R}}

\newcommand{\PC}{\mathscr{P}} 
\newcommand{\FC}{\mathscr{F}} 
 
\DeclareMathOperator{\spt}{spt}
\DeclareMathOperator{\Lip}{Lip}
\DeclareMathOperator{\dist}{dist}
\DeclareMathOperator{\diam}{diam}
\DeclareMathOperator{\ap}{ap} 
\newcommand{\mr}{\mathop{\vrule height 1.6ex depth 0pt width
		0.13ex\vrule height 0.13ex depth 0pt width 1.3ex}\nolimits}
\newcommand{\ora}[1]{{\scriptstyle\xrightarrow{#1}}}
\begin{document}
\title{An inequality on the mass of image of rectifiable chain under chain map}
\author{Chunyan Liu}
\address{CHUNYAN LIU, SCHOOL OF MATHEMATICS AND STATISTICS,
		HUAZHONG UNIVERSITY OF SCIENCE
		AND TECHNOLOGY, 430074, WUHAN, P.R. CHINA}
\email{chunyanliu@hust.edu.cn}
\renewcommand{\thefootnote}{\fnsymbol{footnote}} 
\renewcommand{\thefootnote}{\arabic{footnote}} 
\date{}
\maketitle
\begin{abstract}
For a complete normed abelian group $G$, we show that the mass of image of a rectifiable $G$-chain $S$ under chain map $f_{\sharp}$ induced by Lipschitz map $f$ is controlled by the 
integral of Jacobi of $f$ restricted on the support of $S$ with respect to associated radon mesure $\mu_S$.
\end{abstract}
\section{Introduction}
In our previous paper \cite{Liu:2021}, we proved in Lemma (3.5)
that an inequality on the mass of the image of the rectifiable
$d$-chain $S$ under chain map $f_{\sharp}$, that is, 
\begin{equation}\label{Chan}
\mass(f_{\sharp}S)\leq\int\ap J_d(f|_{M})d\mu_S(x),
\end{equation}
where $f:\mathbb{R}^n\to\mathbb{R}^n$ is of $C^1$, $\mu_S$ is the 
Radon measure associated to chain $S$, and $M$ is a $d$-rectifiable set such that $\mu_S(\mathbb{R}^n\setminus M)=0$. We
further assumed that $P$ is a polyhedral $d$-chain, the inequality
\[
\mass(f_{\sharp}P)\leq\int \ap J_d(f|_{\spt P})(x)d\mu_P(x)
\]
holds for any Lipschitz map $f:\mathbb{R}^n\to\mathbb{R}^n$.
The estimation of mass of the image of flat chain with finite mass 
under induced chain map is important in dealing with the related problems of geometric measure theory. Given a flat $d$-chain $S$
with finite mass, a consequence on the mass of chain $f_{\sharp}S$ under chain map induced by Lipschitz map $f$ in \cite{Fleming:1966} provided that
\[
\mass(f_{\sharp}S)\leq\int\varphi^dd\mu_S(x)
\]
if $||Df(x)||\leq\varphi(x)$, where $\varphi$ is continuous on $\mathbb{R}^n$ and $Df$ denotes the differential. The most commonly
used form of this consequence is 
\[
\mass(f_{\sharp}S)\leq\Lip(f)^d\mass(S).
\]
The above inequality established the relationship between the
mass of $f_{\sharp}S$ under chain map and the mass of the original chain $S$. Since this estimation of the mass of flat $d$-chain $f_{\sharp}S$ is too large for some problems, the
representation of inequality \ref{Chan} is necessary. In this paper, our goal is to prove that 
the inequality \ref{Chan} is also true if we change any Lipschitz map by $C^1$ map. The main theorem of the paper is the following.
\begin{theorem}\label{theorem1}
	Let $f$ be a Lipschitz map of $\mathbb{R}^n$ into $\mathbb{R}^n$, and let $S\in\RC_d(\mathbb{R}^n;G)$. Suppose that $M=\spt S$, then
	\[
	\mass(f_{\sharp}S)\leq\int_{M}\ap J_d(f|_{M})(x)d\mu_S(x).
	\]
\end{theorem}
In order to prove this theorem, we collect some definitions and basic results that we use throughout the paper. For further facts we refer the reader to 
the books by Royden \cite{royden2017real}, Whitney \cite{Whitney:1957}, Federer \cite{Federer:1969} and the article by  Fleming \cite{Fleming:1966}. 

Let $d$ be a positive integer.
For any set $E\subseteq \mathbb{R}^n$, the  
{\it $d$-dimensional Hausdorff measure}
$\HM^d(E)$ is defined by 
\[
\HM^d(E)=\lim\limits_{\delta\to0}\inf
\left\{\sum\diam(U_{i})^d:E \subseteq\bigcup
U_{i}, U_i\subseteq\mathbb{R}^n, \diam(U_i)\leq\delta\right\}.
\]

A {\it normed abelian group} is an abelian group $G$ equipped with a norm $|.|:G\to[0,\infty)$ satisfying
\begin{enumerate}
\item $\left|-g\right|=\left|g\right|$,
\item $\left|g+h\right|\leq\left|g\right|+\left|h\right|$,
\item $\left|g\right|=0\text{ if and only if }g=0$.
\end{enumerate}
A normed abelian group $G$ is {\it complete} if it is complete 
with respect to the metric induced by the norm.

The group of 
{\it polyhedral chains} of dimension $d$ 
in $\mathbb{R}^n$, with coefficients in $G$, denoted by  $\PC_d(\mathbb{R}^n;G)$, is a collection of elements consist of
$\sum_{i=1}^n g_i\Delta_i$ with $g_i\in G$ 
and $\Delta_i$
are interior disjoint polyhedra of dimension $d$. The {\it mass} of $P$ 
is defined by
\[
\mass(P)=\sum\limits_{i=1}^{n}\left|g_i\right|
\HM^d(\Delta_i).
\]
The  {\it Whitney flat norm} on $\PC_d(\mathbb{R}^n;G)$ 
is defined by 
\[
\fn(P)=\inf\left\{\mass(Q)+\mass(R):P=Q+\partial
R,Q\in\PC_d(\mathbb{R}^n;G),R\in\PC_{d+1}
(\mathbb{R}^n;G)\right\}.
\]
The group of {\it flat chains} of dimension $d$, $\FC_d(\mathbb{R}^n;G)$, is the completion of 
$\PC_d(\mathbb{R}^n;G)$ with respect to the Whitney flat norm $\fn$. For any $S\in\FC_d(\mathbb{R}^n;G)$, we denote by $P_i\ora\fn S$ a sequence of polyhedral chains $\{P_i\}$ converging to $S$ in Whitney flat norm
$\fn$. Then the mass of $S$ is defined by 
\[
\mass(S)=\inf\left\{\liminf\limits_{i\to\infty}
\mass(P_i):P_i\ora\fn S,P_i\in\PC_d(\mathbb{R}^n;G)
\right\}.
\]

Associated to every finite mass $d$-chain $S$ and Borel set $E$, there
is a finite mass $d$-chain $S\mr E$ that is, roughly speaking, the portion 
of $S$ in $E$. Then $E\mapsto\mass(S\mr E)$ defines {\it a Radon measure $\mu_S$} such that 
\[
\mu_S(E)=\mass(S\mr E).
\]
A flat chain $S$ is
supported by a closed set $X$ if for every open set $U$
containing $X$, there is a sequence of polyhedral chains 
$\left\{P_i\right\}$ tending to $S$ with cells of each $P_i$ contained in $U$. The {\it support of $S$}, denoted by $\spt S$, is the smallest closed set $X$ supporting $S$, if it exits.

A Lipschitz map $f:\mathbb{R}^n\to\mathbb{R}^n$ induceds a homomorphic chain map $f_{\sharp}:\FC_d(\mathbb{R}^n;G)\to\FC_d(\mathbb{R}^n;G)$. If $P$ is 
a polyhedral $d$-chain, then $f_{\sharp}P$ is a Lipschitz $d$-chain, defined in \cite[p.297]{Whitney:1957} by approximating $f$ by piecewise-linear functions.

A flat chain $S\in\FC_d(\mathbb{R}^n;G)$ is called 
{\it rectifiable} if for each
$\varepsilon>0$ there exists a polyhedral $d$-chain 
$P\in \PC_d(\mathbb{R}^n;G)$ 
and a Lipschitz map 
$f:\mathbb{R}^n\to\mathbb{R}^n$ such that 
\[
\mass(S-f_{\sharp}P)<\varepsilon.
\]
We denote by $\RC_d(\mathbb{R}^n;G)$ the collection of 
all rectifiable $d$-chains.

For any Lipschitz mapping $f: X\to\mathbb{R}^n$, $X\subseteq \mathbb{R}^n$, the approximate $d$ dimensional Jacobian of $f$ is
defined by the formula
\[
\ap J_df(x)=||\wedge_d\ap Df(x)||
\]
whenever $f$ is approximately differentiable at $x$.

\section{Proof of theorem \ref{theorem1}}
First, we will introduce the
following approximation theorem in \cite{PauwV:2018} by 
De Pauw.
\begin{theorem}(Approximation theorem)\label{bijin}
Let $S\in\RC_d(\mathbb{R}^n;G)$ and $\varepsilon>0$. There exist
$Q\in\PC_d(\mathbb{R}^n;G)$ and $g:\mathbb{R}^n\to\mathbb{R}^n$ a 
diffeomorphism of class $C^1$ with the following properties:
\begin{enumerate}
	\item $\max\{\Lip(g),\Lip(g^{-1})\}\leq1+\varepsilon$,
	\item $|g(x)-x|\leq\varepsilon$  for verery $x\in\mathbb{R}^n$,
	\item $g(x)=x$ whenever $\dist(x,\spt S)\ge\varepsilon$,
	\item $\mass(Q-g_{\sharp}S)\leq\varepsilon$.
\end{enumerate}
\end{theorem}
A signed measure $\mu$ on the measurable space $(\mathbb{R}^n,\mathcal{M})$ can be decomposed into two mutually singular
measure $\mu^{+}$ and $\mu^{-}$ on $(\mathbb{R}^n,\mathcal{M})$ such that $\mu=\mu^{+}-\mu^{-}$, where $\mathcal{M}$ is a $\sigma$-algebra of subsets of $\mathbb{R}^n$. The measure $|\mu|$is defined on $\mathcal{M}$ by 
\[
|\mu|(E)=\mu^{+}(E)+\mu^{-}(E) \text{ for all }E\in\mathcal{M}.
\]
$|\mu|(\mathbb{R}^n)$ is called the total variation of $\mu$
and denoted by $||\mu||_{TV}$, that is,
\[
||\mu||_{TV}=|\mu|(\mathbb{R}^n)=\sup\sum_{k=1}^{m}|\mu(E_k)|,
\]
where the supremum is taken over all finite disjoint collection $\{E_k\}_{k=1}^m$ of measurable subsets of $\mathbb{R}^n$.

We start with the proof of the following result on the total variation of measure.
\begin{lemma}\label{quan}
	Let $S\in\RC_d(\mathbb{R}^n;G)$. For any $0<\varepsilon<1$, 
there exist $Q_{\varepsilon}\in\PC_d(\mathbb{R}^n;G)$ and
$g_{\varepsilon}:\mathbb{R}^n\to\mathbb{R}^n$ 
diffeomorphism of class $C^1$ such that
\[
||{g^{-1}_{\varepsilon}}_{\sharp}\mu_{Q_{\varepsilon}}-\mu_S||_{TV}\leq(1+2^d\mass(S))\varepsilon
\]
\end{lemma}
\begin{proof}
Since $S\in\RC_d(\mathbb{R}^n;G)$, by Lemma \ref{bijin}, we get that for any $0<\varepsilon<1$, there exist $Q_{\varepsilon}\in\PC_d(\mathbb{R}^n;G)$ and 
$g_{\varepsilon}:\mathbb{R}^n\to\mathbb{R}^n$ diffeomorphism of class $C^1$ such that 
\[
\max\{\Lip(g_{\varepsilon}), \Lip(g_{\varepsilon}^{-1})\}\leq1+\varepsilon, \mass(Q_{\varepsilon}-{g_{\varepsilon}}_{\sharp}S)\leq\varepsilon.
\]
Thus for any measurable set $E$,
\[
\mass({g_{\varepsilon}}_{\sharp}(S\mr E))\leq\Lip(g_{\varepsilon})^d\mass(S\mr E)\leq
(1+\varepsilon)^d\mass(S\mr E),
\]
\[
\mass(S\mr E)=\mass({g^{-1}_{\varepsilon}}_{\sharp}({g_{\varepsilon}}_{\sharp}(S\mr E)))\leq
(1+\varepsilon)^d\mass({g_{\varepsilon}}_{\sharp}(S\mr E)),
\]
that is,
\[
(1/(1+\varepsilon)^d-1)\mass(S\mr E)\leq\mass({g_{\varepsilon}}_{\sharp}(S\mr E))-\mass(S\mr E)\leq((1+\varepsilon)^d-1)\mass(S\mr E).
\]
Since 
\[
1-1/(1+\varepsilon)^d=((1+\varepsilon)^d-1)/(1+\varepsilon)^d<(1+\varepsilon)^d-1,
\]
thus
\[
\big|\mass({g_{\varepsilon}}_{\sharp}(S\mr E))-\mass(S\mr E)\big|\leq((1+\varepsilon)^d-1)\mass(S\mr E).
\]
We see that
\[\begin{aligned}
||{g^{-1}_{\varepsilon}}_{\sharp}\mu_{Q_{\varepsilon}}-\mu_S||_{TV}&\leq||{g^{-1}_{\varepsilon}}_{\sharp}(\mu_{Q_{\varepsilon}}-\mu_
{{g_{\varepsilon}}_{\sharp}S})||_{TV}+||{g^{-1}_{\varepsilon}}_{\sharp}\mu_{{g_{
			\varepsilon}}_{\sharp}S}-\mu_S||_{TV}
\\&=\big|{g^{-1}_{\varepsilon}}_{\sharp}(\mu_{Q_{\varepsilon}}-\mu_{{g_{\varepsilon}}_{\sharp}S})\big|(\mathbb{R}^n)+
\big|{g^{-1}_{\varepsilon}}_{\sharp}\mu_{g_{\varepsilon\sharp}S}-\mu_S\big|(\mathbb{R}^n)\\&=
\sup\sum_{k=1}^{m}\big|{g^{-1}_{\varepsilon}}_{\sharp}(\mu_{Q_{\varepsilon}}-\mu_{{g_{
			\varepsilon}}_{\sharp}S})(E_k)\big|+\sup\sum_{k=1}^m
\big|({g^{-1}_{\varepsilon}}_{\sharp}\mu_{{g_{\varepsilon}}_{\sharp}S}-\mu_S)(E_k)\big|\\&
=\sup\sum_{k=1}^{m}\big|(\mu_{Q_{\varepsilon}}-\mu_{{g_{\varepsilon}}_{\sharp}S})
(E_k')\big|+\sup\sum_{k=1}^m\big|\mu_{{g_{\varepsilon}}_{\sharp}S}(E_k')-\mu_S(E_k)\big|\\&
=\sup\sum_{k=1}^{m}\big|\mass(Q_{\varepsilon}\mr E_k')-\mass(({g_{\varepsilon}}_{\sharp}S)\mr E_k')\big|+\sup\sum_{k=1}^m\big|\mass({g_{\varepsilon}}_{\sharp}(S\mr E_k))-\mass(S\mr E_k)\big|\\&
\leq\sup\sum_{k=1}^m\mass((Q_{\varepsilon}-{g_{\varepsilon}}_{\sharp}S)\mr E_k')+\sup\sum_{k=1}^m((1+\varepsilon)^d-1)\mass(S\mr E_k)
\\&\leq\mass(Q_{\varepsilon}-{g_{\varepsilon}}_{\sharp}S)+
((1+\varepsilon)^d-1)\mass(S)\\&
\leq\varepsilon+2^d\mass(S)\varepsilon,
\end{aligned}\]
where the supremum is taken over all finite disjoint collection 
$\{E_k\}_{k=1}^m$ of measurable subsets of $\mathbb{R}^n$,
$E_k'=g_{\varepsilon}(E_k)$. 
\end{proof}
We are now in a position to prove the main theorem of this paper.
\begin{proof}[Proof of Theorem \ref{theorem1}]
Let $g_{\varepsilon}$ be diffeomorphism of class $C^1$ and $Q_{\varepsilon}$ be polyhedral $d$-chain as above Lemma \ref{quan}.
\[
f_{\sharp}S=f_{\sharp}(g_{\varepsilon}^{-1}\circ g_{\varepsilon})_{\sharp}S=f_{\sharp}\circ {g^{-1}_{\varepsilon}}_{\sharp}({g_{\varepsilon}}_{\sharp}S),
\]
\[
{g_{\varepsilon}}_{\sharp}S=Q_{\varepsilon}+{g_{\varepsilon}}_{\sharp}S-Q_{\varepsilon},
\]
thus we get that
\[\begin{aligned}
\mass(f_{\sharp}S)&=\mass(f_{\sharp}\circ {g^{-1}_{\varepsilon}}_{\sharp}({g_{\varepsilon}}_{\sharp}S))
=\mass(f_{\sharp}\circ {g^{-1}_{\varepsilon}}_{\sharp}(Q_{\varepsilon}+{g_{\varepsilon}}_{\sharp}S-Q_{\varepsilon}))\\&
\leq\mass(f_{\sharp}\circ {g^{-1}_{\varepsilon}}_{\sharp}(Q_{\varepsilon}))+\mass(f_{\sharp}\circ {g^{-1}_{\varepsilon}}_{\sharp}({g_{\varepsilon}}_{\sharp}S-Q_{\varepsilon})).
\end{aligned}
\]
Since
\[
\big|\big|\wedge_d\ap D(g_{\varepsilon}^{-1}|_{\spt Q_{\varepsilon}})(x)\big|\big|\leq\big|\big| \ap D(g_{\varepsilon}^{-1}|_{\spt Q_{\varepsilon}})(x)\big|\big|^d\leq\Lip
(g_{\varepsilon}^{-1})^d\leq(1+\varepsilon)^d,
\]
we know that 
\[\begin{aligned}
\mass(f_{\sharp}\circ {g^{-1}_{\varepsilon}}_{\sharp}(Q_{\varepsilon}))&\leq\int_{\spt Q_{\varepsilon}}\ap J_d((f\circ g_{\varepsilon}^{-1})|_{\spt Q_{\varepsilon}})(x)d\mu_{Q_{\varepsilon}}(x)\\&
=\int_{\spt Q_{\varepsilon}}\ap J_d(f|_{g_{\varepsilon}^{-1}(\spt Q)})(g_{\varepsilon}^{-1}(x))\cdot\ap J_d(g_{\varepsilon}^{-1}|_{\spt Q_{\varepsilon}})(x)d\mu_{Q_{\varepsilon}}(x)\\&\leq
\int_{\spt Q_{\varepsilon}}\big|\big|\wedge_d\ap D(f|_{g_{\varepsilon}^{-1}(\spt Q)})(g_{\varepsilon}^{-1}(x))\big|\big|\cdot\big|\big|\wedge_d \ap D(g_{\varepsilon}^{-1}|_{\spt Q_{\varepsilon}})(x)\big|\big|d\mu_{Q_{\varepsilon}}(x)\\&\leq
(1+\varepsilon)^d\int_{\spt Q_{\varepsilon}}\big|\big|\wedge_d\ap D(f|_{g_{\varepsilon}^{-1}(\spt Q)})(g_{\varepsilon}^{-1}(x))\big|\big|d\mu_{Q_{\varepsilon}}(x)\\&
=(1+\varepsilon)^d\int_{g_{\varepsilon}^{-1}(\spt Q_{\varepsilon})}\ap J_d(f|_{
	g_{\varepsilon}^{-1}(\spt Q_{\varepsilon})})(x)d{g_{\varepsilon}^{-1}}_{\sharp}\mu_{Q_{\varepsilon}}(x),
\end{aligned}\]
Set 
\[
M_{\varepsilon}=g_{\varepsilon}^{-1}(\spt Q_{\varepsilon}), M_S=\spt S, \mu_{\varepsilon}={g_{\varepsilon}^{-1}}_{\sharp}
\mu_{Q_{\varepsilon}},
\]
\[\begin{aligned}
\int_{M_{\varepsilon}}\ap J_d(f|_{M_{\varepsilon}})(x)
d\mu_{\varepsilon}(x)&\leq\int_{\mathbb{R}^n}\ap J_d(f|_{M_{\varepsilon}})d(\mu_{\varepsilon}-\mu_S)(x)+\int_{\mathbb{R}^n}\ap J_d(f|_{M_{\varepsilon}})d\mu_S(x)\\&\leq
\int_{\mathbb{R}^n}\ap J_d(f|_{M_{\varepsilon}})d|\mu_{\varepsilon}-\mu_S|(x)+\int_{\mathbb{R}^n}\ap J_d(f|_{M_S})d\mu_S(x).
\end{aligned}\]
By Lemma \ref{quan}, 
\[\begin{aligned}
\int_{\mathbb{R}^n}\ap J_d(f|_{M_{\varepsilon}})d|\mu_{\varepsilon}-\mu_S|(x)&\leq\Lip(f)^d
|\mu_{\varepsilon}-\mu_S|(\mathbb{R}^n)\\&\leq\Lip(f)^d(1+
2^d\mass(S))\varepsilon.
\end{aligned}\] 
Hence $\varepsilon\to0$, 
\[
\mass(f_{\sharp}\circ {g^{-1}_{\varepsilon}}_{\sharp}(Q_{\varepsilon}))\leq\int_{\mathbb{R}^n}\ap J_d(f|_{M_S})d\mu_S(x)=\int_{M_S}\ap J_d(f|_{M_S})d\mu_S(x).
\]
Moreover, when $\varepsilon\to0$,
\[
\mass(f_{\sharp}\circ {g^{-1}_{\varepsilon}}_{\sharp}({g_{\varepsilon}}_{\sharp}S-Q_{\varepsilon}))\leq\Lip(f)^d(1+\varepsilon)^d\mass({g_{\varepsilon}}_{\sharp}S-Q_{\varepsilon})\leq\Lip(f)^d(1+\varepsilon)^d\varepsilon\to0.
\]
Therefore we get that
\[
\mass(f_{\sharp}S)\leq\int_{M_S}\ap J_d(f|_{M_S})d\mu_S(x).
\]
\end{proof}
\section{Acknowledgement}
I would like to express my gratitude to my supervisor. Dr. Yangqin Fang for fruitful discussion.
\bibliography{1}
\end{document}